\documentclass{amsart}

\usepackage{latexsym,amsmath,amsfonts,amscd,amssymb, amsthm}

\theoremstyle{plain}
\newtheorem{theorem}{Theorem}[section]

\newtheorem{proposition}{Proposition}[section]
\newtheorem{corollary}{Corollary}[section]
 
\theoremstyle{definition}

\newtheorem{example}{Example}[section]

\theoremstyle{remark}
\newtheorem{remark}{Remark}[section]

\title{Invariant CR Mappings}

\author{John P. D'Angelo}

\address{Dept. of Mathematics, Univ. of Illinois, 1409 W. Green St., Urbana IL 61801}

\email{jpda@math.uiuc.edu}

\begin{document}

\maketitle

\medskip
{\it Dedicated to Linda Rothschild}
\medskip

\section{Introduction}

The subject of CR Geometry interacts with nearly all of mathematics.
See  [BER] for an extensive discussion of many aspects of CR manifolds and mappings between them.
One aspect of the subject not covered in [BER] concerns CR mappings invariant under groups.
The purpose of this paper is to discuss interactions with number theory and combinatorics
that arise from the seemingly simple setting of group-invariant CR mappings from the unit sphere
to a hyperquadric. Some elementary representation theory also arises.

The unit sphere $S^{2n-1}$ in complex Euclidean space ${\bf C}^n$ is the basic example of a CR manifold of hypersurface type.
More generally we consider the hyperquadric $Q(a,b)$ defined to be the subset of ${\bf C}^{a+b}$ defined by

$$ \sum_{j=1}^a |z_j|^2 - \sum_{j=a+1}^{a+b} |z_j|^2 = 1. \eqno (1) $$

Of course $S^{2n-1}$ is invariant under the unitary group $U(n)$. Let $\Gamma$ be a finite subgroup
of $U(n)$. Assume that $f:{\bf C}^n \to {\bf C}^N$ is a rational mapping invariant under $\Gamma$ and that $f(S^{2n-1}) \subset S^{2N-1}$.
For most $\Gamma$ such an $f$ must be a constant. In other words, for most $\Gamma$,
there is no non-constant $\Gamma$-invariant rational mapping from sphere to sphere for any target dimension.
In fact, for such a non-constant invariant map to exist, $\Gamma$ must be cyclic and represented in a rather restricted fashion.
See [Li1], [DL], and especially Corollary 7 on page 187 of [D1], for precise statements and the considerable details required.

The restriction to rational mappings is natural; for $n\ge 2$ Forstneric [F1] proved that a proper mapping between balls,
with sufficiently many continuous derivatives at the boundary, must be a rational mapping. On the other hand, if one makes no regularity
assumption at all on the map, then (see [Li2]) one can create group-invariant proper mappings between balls for any fixed-point free
finite unitary group.  The restrictions on the group arise from CR Geometry and the smoothness of the CR mappings considered.
In this paper we naturally restrict our considerations to the class of rational mappings. See [F2] for considerable discussion 
about proper holomorphic mappings and CR Geometry.

In order to find group-invariant CR mappings from a sphere,
we relax the assumption that the target manifold be a sphere, and instead allow it to be a hyperquadric.
We can then always find polynomial examples, as we note in Corollary 1.1.
In this paper we give many examples of invariant mappings from spheres to hyperquadrics. Our techniques
allow us to give some explicit surprising examples. In Theorem 6.1 for example, we show 
that {\it rigidity} fails for mappings between hyperquadrics;
we find non-linear polynomial mappings between hyperquadrics with the same number of negative eigenvalues
in the defining equations of the domain and target hyperquadrics. As in the well-known case of maps between spheres,
we must allow sufficiently many positive eigenvalues in the target for such maps to exist.

To get started we recall that a polynomial $R:{\bf C}^n \times {\bf C}^n \to {\bf C}$ is called Hermitian symmetric 
if $R(z,{\overline w}) = {\overline {R(w, {\overline z})}}$
for all $z$ and $w$. If $R$ is Hermitian symmetric, 
then $R(z, {\overline z})$ is evidently real-valued. By polarization, the converse also holds. We note also
that a polynomial in $z=(z_1,...,z_n)$ and ${\overline z}$ is Hermitian symmetric if and only if its matrix of coefficients
is Hermitian symmetric in the sense of linear algebra.

The following result from [D1] shows how to construct group-invariant mappings from spheres
to hyperquadrics. Throughout the paper we will give explicit formulas in many cases.

\begin{theorem} Let $\Gamma$ be a finite subgroup of $U(n)$ of order $p$. Then there is a unique Hermitian symmetric
$\Gamma$-invariant polynomial $\Phi_\Gamma(z, {\overline w})$ such that the following hold:

1) $\Phi_\Gamma(0,0) = 0$. 

2) The degree of $\Phi_\Gamma$ in $z$ is $p$.

3) $\Phi_\Gamma(z, {\overline z})= 1$ when $z$ is on the unit sphere.

4) $\Phi_\Gamma(\gamma z, {\overline w}) = \Phi_\Gamma(z, {\overline w})$ for all $\gamma \in \Gamma$. \end{theorem}

\begin{corollary} There are holomorphic vector-valued $\Gamma$-invariant polynomial mappings $F$ and $G$
such that we can write

$$ \Phi_\Gamma(z, {\overline z})= ||F(z)||^2 - ||G(z)||^2. \eqno (2) $$ 
The polynomial mapping $z \to (F(z), G(z))$ restricts to a $\Gamma$-invariant mapping from $S^{2n-1}$
to the hyperquadric $Q(N_+, N_-)$, where these integers are the numbers of positive and negative eigenvalues
of the matrix of coefficients of $\Phi_\Gamma$. \end{corollary}

The results in this paper revolve around how the mapping $(F,G)$ from Corollary 1.1 depends on $\Gamma$.
We clarify one point at the start; even if we restrict our considerations to cyclic groups, then this mapping changes
(surprisingly much) as the representation of the group changes. The interesting things from the points of view of CR Geometry,
Number Theory, and Combinatorics all depend in non-trivial ways on the particular representation.
Therefore the results should be considered as statements about the
particular subgroup $\Gamma \subset U(n)$, rather than as statements about the abstract group $G$ for which
$\pi(G) = \Gamma$.

The proof of Theorem 1 leads to the following formula for $\Phi_\Gamma$.

$$ \Phi_\Gamma(z,{\overline w}) = 1 - \prod_{\gamma \in \Gamma}(1 - \langle \gamma z, w \rangle). \eqno (3) $$
The first three properties from Theorem 1 are evident from (3), and the fourth property is not hard to check.
One also needs to verify uniqueness.

The starting point for this paper will therefore be formula (3). 
We will first consider three different representations of cyclic groups and we note the considerable
differences in the corresponding invariant polynomials. We also consider metacyclic groups.
We also discuss some interesting asymptotic considerations, as the order of the group tends to infinity. Additional
asymptotic results are expected in appear in the doctoral thesis [G] of D. Grundmeier.

An interesting result in this paper is the application in Section 6. In Theorem 6.1
we construct, for each odd $2p+1$ with $p \ge 1$, a  polynomial mapping $g_p$ of degree $2p$ such that
$$ g_p: Q(2,2p+1) \to Q(N(p), 2p+1). \eqno (4) $$
These mappings illustrate a failure of {\it rigidity}; in many contexts restrictions on the eigenvalues of the domain
and defining hyperquadrics force maps to be linear. See [BH]. Our new examples show that rigidity
does not hold when we keep the number of negative eigenvalues the same, as long as we allow a sufficient increase
in the number of positive eigenvalues. On the other hand, by a result in [BH], the additional restriction
that the mapping preserves sides of the hyperquadric does then guarantee rigidity. 
It is quite striking that the construction of the polynomials in Theorem 6.1 relies on
the group-theoretic methods in the rest of the paper. 
 
The author acknowledges support from NSF grant DMS-07-53978. He thanks both Dusty Grundmeier and Jiri Lebl
for many discussions on these matters. He also acknowledges the referee who spotted an error in the original presentation
of Example 3.3.

\section{Properties of the invariant polynomials}

Let $\Gamma$ be a finite subgroup of the unitary group $U(n)$, and let $\Phi_\Gamma$ be
the unique polynomial defined by (3). Our primary interest concerns how this polynomial
depends on the particular representation of the group.

We remark at the outset that we will be considering {\it reducible}
representations. A simple example clarifies why. If $G$ is cyclic of order $p$, then $G$ has the irreducible unitary (one-dimensional)
representation $\Gamma$ as the group of $p$-th roots of unity. An elementary calculation shows
that the invariant polynomial $\Phi_\Gamma$ becomes simply 
$$ \Phi_\Gamma(z,{\overline w}) = (z{\overline w})^p. \eqno (5) $$
On the other hand, there are many ways to represent $G$ as a subgroup of $U(n)$ for $n\ge 2$. We will consider these below;
for now we mention one beautiful special case.

Let $p$ and $q$ be positive integers with $1\le q\le p-1$ and let $\omega$ be a primitive $p$-th
root of unity. Let $\Gamma(p,q)$ be the cyclic group generated by the diagonal $2$-by-$2$ matrix $A$ with eigenvalues
$\omega$ and $\omega^q$:

$$ 
A= \begin{pmatrix}
\omega     & 0 \\
0 & \omega^q
\end{pmatrix}.
 \eqno (6) $$
Because $A$  is diagonal, the invariant polynomial $\Phi_{\Gamma(p,q)}(z, {\overline z})$ depends on only $|z_1|^2$ and $|z_2|^2$.
If we write $x=|z_1|^2$ and $y=|z_2|^2$, then we obtain a corresponding polynomial $f_{p,q}$ in $x$ and $y$.
This polynomial has integer coefficients; a combinatorial interpretation of these coefficients appears in [LWW].
The crucial idea in [LWW] is the interpretation of $\Phi_\Gamma$ as a circulant determinant; hence permutations
arise and careful study of their cycle structure leads to the combinatorial result.
Asymptotic information about these integers as $p$ tends to infinity appears in both [LWW] and [D4]; the technique
in [D4] gives an analogue of the Szeg\"o limit theorem.  In the special case where $q=2$, these polynomials provide 
examples of sharp degree estimates for proper monomial mappings between balls. The polynomials $f_{p,2}$
have many additional beautiful properties. We pause to write down the formula and state an appealing corollary.
These polynomials will arise in the proof of Theorem 6.1.

$$  f_{p,2}(x,y) = (-1)^{p+1} y^p + \left( {x + \sqrt{x^2 + 4y} \over 2}\right)^p + \left( {x - \sqrt{x^2 + 4y} \over 2}\right)^p. \eqno (7) $$ 

\begin{corollary} [D4] Let $S_p$ be the sum of the coefficients of $f_{p,2}$. Then the limit as $p$ tends to infinity
of $S_p^{1 \over p}$ equals the golden ratio ${1 + \sqrt{5} \over 2}$. \end{corollary}

\begin{proof} The sum of the coefficients is $f_{p,2}(1,1)$, so put $x=y=1$ in (7).
The largest (in absolute value) of the three terms 
is the middle term. Taking $p$-th roots and letting $p$ tend to infinity gives the result. \end{proof}

See [DKR] for degree estimates and [DLe] for number-theoretic information concerning uniqueness
results for degree estimates. The following elegant primality test was proved in [D2].

\begin{theorem}  For each $q$, the congruence $f_{p,q}(x,y) \cong x^p + y^p$ mod $(p)$ holds if and only if $p$ is prime.     \end{theorem}

We make a few comments. When $q=1$, the polynomial $f_{p,1}$ is simply $(x+y)^p$ and the result is well-known.
For other values of $q$ the polynomials are more complicated. When $q=2$ or when $q=p-1$ there are explicit formulas
for the integer coefficients. For small $q$ recurrences exist but the order of the recurrences grows exponentially with $q$. 
See [D2], [D3], [D4] and [G].
There is no known general formula for the integer coefficients.  Nonetheless the basic theory enables us to reduce
the congruence question to the special case. 
Note also that the quotient space $L(p,q)=S^{3}/ \Gamma$ is a Lens space. It might be interesting to relate
the polynomials $f_{p,q}$ to the differential topology of these spaces.

We return to the general situation and
repeat the crucial point; the invariant polynomials depend on the representation in non-trivial and interesting ways,
even in the cyclic case. In order to express them we recall ideas that go back to E. Noether. See [S] for considerable discussion.
Given a subgroup $\Gamma$ of the general linear group, Noether proved that the algebra of polynomials invariant under $\Gamma$
is generated by polynomials of degree at most the order $|\Gamma|$ of $\Gamma$. Given a polynomial $p$ we can create
an invariant polynomial by averaging $p$ over the group: 
$$ {1 \over |\Gamma|} \sum_{\gamma \in \Gamma} p \circ \gamma. \eqno (8) $$
We find a basis for the algebra of invariant polynomials as follows. We average each monomial $z^\alpha$
of total degree at most $|\Gamma|$ as in (8) to obtain an invariant polynomial; often the result will be the zero polynomial.
The nonzero polynomials that result generate the algebra of polynomials invariant under $\Gamma$. In particular, the number
of polynomials required is bounded above by the dimension of the space of homogeneous polynomials of degree $|\Gamma|$
in $n$ variables. Finally we can express the $F$ and $G$ from (2) in terms of sums and products of these basis elements.
The invariant polynomials here are closely related to the Chern orbit polynomials from [S]. The possibility of polarization
makes our approach a bit different. It seems a worthwhile project to deepen this connection. Some results in this direction
will appear in [G].

\section{Cyclic groups}
Let $\Gamma$ be cyclic of order $p$. Then the elements of $\Gamma$ are $I,A,A^2,...,A^{p-1}$ for some unitary
matrix $A$. Formula (3) becomes

$$ \Phi_\Gamma(z,{\overline w}) = 1 - \prod_{j=0}^{p-1} (1 - \langle A^j z, w \rangle). \eqno (9) $$
We begin by considering three different representations of
a cyclic group of order six; we give precise formulas for the corresponding invariant polynomials.

Let $\omega$ be a primitive sixth-root of unity, and let $\eta$ be a primitive third-root of unity.
We consider three different unitary matrices; each generates a cyclic group of order six.

\begin{example} Let $\Gamma$ be the cyclic group of order $6$ generated by $A$, where
$$ 
A= \begin{pmatrix}
\omega     & 0 \\
0 & {\omega}
\end{pmatrix}.
 \eqno (10.1) $$
The invariant polynomial satisfies the following formula:
$$ \Phi_\Gamma(z, {\overline z}) = (|z_1|^2 + |z_2|^2)^6. \eqno (10.2) $$
It follows that $\Phi$ is the squared norm of the following holomorphic polynomial:

$$ f(z) = (z_1^6, \sqrt{6} z_1^5 z_2, \sqrt{15} z_1^4  z_2^2, \sqrt{20} z_1^3 z_2^3, \sqrt{15} z_1^2 z_2^4, \sqrt{6} z_1 z_2^5, z_2^6). \eqno (10.3) $$

The polynomial $f$ restricts to the sphere to define an invariant CR mapping from $S^3$ to $S^{13} \subset {\bf C}^7$. \end{example}

\begin{example} Let $\Gamma$ be the cyclic group of order $6$ generated by $A$, where
$$ 
A= \begin{pmatrix}
\omega     & 0 \\
0 & {\overline \omega}
\end{pmatrix}.
 \eqno (11.1) $$
The invariant polynomial satisfies the following formula:
$$ \Phi_\Gamma(z, {\overline z}) = |z_1|^{12} + |z_2|^{12} + 6 |z_1|^2 |z_2|^2 + 2 |z_1|^6 |z_2|^6 - 9 |z_1|^4 |z_2|^4.  \eqno (11.2) $$
Note that $\Phi$ is not a squared norm. Nonetheless we define $f$ as follows:

$$ f(z) = (z_1^6, z_2^6, \sqrt{6} z_1 z_2, \sqrt{2} z_1^3 z_2^3, 3 z_1^2 z_2^2). \eqno (11.3) $$
Then
$$ \Phi = |f_1|^2 + |f_2|^2 + |f_3|^2 +|f_4|^2 - |f_5|^2, $$
and the polynomial $f$ restricts to the sphere to define an invariant CR mapping from $S^3$ to $Q(4,1) \subset {\bf C}^5$. \end{example}

\begin{example} Let $\Gamma$ be the cyclic group of order $6$ generated by $A$, where
$$ 
A= \begin{pmatrix}
0 & 1 \\
\eta & 0 
\end{pmatrix}.
 \eqno (12.1) $$

The polynomial $\Phi$ can be expressed as follows:

$$ \Phi(z,{\overline z}) = (x+y)^3 + (\eta s + {\overline \eta} t)^3 - (x+y)^3 (\eta s + {\overline \eta} t)^3 \eqno (12.2)$$
where

$$ x = |z_1|^2$$
$$ y = |z_2|^2 $$
$$ s = z_2 {\overline z_1} $$
$$ t = z_1 {\overline z_2}.  $$

After diagonalization this information determines a (holomorphic) polynomial mapping $(F,G)$ such that
$$\Phi_\Gamma = ||F||^2 - ||G||^2. $$
It is somewhat complicated to determine the components of $F$ and $G$.
It is natural to use Noether's approach.
For this particular representation, considerable computation then yields the following invariant polynomials:

$$ z_1^3 + z_2^3 = p $$
$$ z_1^2 z_2 + \eta z_1 z_2^2 = q $$
$$  z_1^6 + z_2^6 = f $$
$$  z_1^3 z_2^3 = {1 \over 2} (p^2 - f)$$
$$  z_1 z_2^5 + \eta^2 z_1^5 z_2$$
$$ {1 \over 2} (z_1^4 z_2^2 + \eta^2 z_1^2 z_2^4) = h $$

In order to write $\Phi_\Gamma$ nicely, we let
$$ g = c (z_1 z_2^5 + 3 \eta z_1^3 z_2^3 + \eta^2 z_1^5 z_2). $$
Then one can write $\Phi_\Gamma$, for some $C>0$ as follows:

$$ \phi_\Gamma = |p|^2 + |q|^2 + {1 \over 2} (|f - z_1^3 z_2^3|^2 - |f + z_1^3 z_2^3|^2) + C (|g -h|^2 - |g + h|^2). \eqno (12.3)$$
We conclude that the invariant polynomial determines an invariant CR mapping $(F,G)$
from the unit sphere $S^3$ to the hyperquadric $Q(4,2) \subset {\bf C}^6$.
We have

$$ F = \left(p,q,{1 \over \sqrt{2}} (f- z_1^3 z_2^3), \sqrt{C} (g-h)\right) \eqno (12.4) $$
$$ G = \left( {1 \over \sqrt{2}} (f+ z_1^3 z_2^3), \sqrt{C} (g+h) \right). \eqno (12.5) $$
\end{example}

Consider these three examples together. In each case we have a cyclic group of order six, represented as a subgroup of $U(2)$.
In each case we found an invariant CR mapping. The image hyperquadrics were $Q(7,0)$, $Q(4,1)$, and $Q(4,2)$. The corresponding 
invariant mappings had little in common. In the first case, the map was homogeneous; in the second case the map was not homogeneous, 
although it was a monomial mapping. In the third case we obtained a rather complicated non-monomial map.
It should be evident from these examples that the mappings depend in non-trivial ways on the representation.

\section{Asymptotic Information}
In this section we consider three families of cyclic groups,
$\Gamma(p,1)$, $\Gamma(p,2)$, and $\Gamma(p,p-1)$. For these groups it is possible to compute
the invariant polynomials $\Phi_\Gamma$ exactly. In each case, because the group is generated by a diagonal matrix,
the invariant polynomial depends on only $x=|z_1|^2$ and $y=|z_2|^2$. We will therefore often write the polynomials
as functions of $x$ and $y$.

For $p=1$ we have
$$ \Phi_\Gamma(z,{\overline z}) = (|z_1|^2 + |z_2|^2)^p = (x+y)^p. \eqno (13) $$
It follows that there is an invariant CR mapping to a sphere, namely the hyperquadric $Q(p+1,0)$.
We pause to prove (13) by establishing the corresponding general result in any domain dimension.

\begin{theorem} Let $\Gamma$ be the cyclic group generated by $\omega I$, where
$I$ is the identity operator on ${\bf C}^n$ and $\omega$ is a primitive $p$-th root of unity.
Then $\Phi_\Gamma(z,{\overline z}) = ||z||^{2p}= ||z^{\otimes p}||^2$.
Thus $\Phi_{\Gamma(p,1)}^{1 \over p}= ||z||^2$ and hence it is independent of $p$. \end{theorem}

\begin{proof} A basis for the invariant polynomials is given by the homogeneous monomials of degree $p$. By Theorem 1.1
$\Phi_\Gamma$ is of degree $p$ in $z$ and hence of degree $2p$ overall. It must then be homogeneous of total degree $2p$
and it must take the value $1$ on the unit sphere; it therefore equals $||z||^{2p}$.  \end{proof}

We return to the case where $n=2$ where $||z||^2 = |z_1|^2 + |z_2|^2 = x+y$.
In the more complicated situation
arising from $\Gamma(p,q)$, the expression $\Phi_{\Gamma(p,q)}^{1 \over p}$ is not constant,
but its behavior as $p$ tends to infinity is completely analyzed in [D4].

As an illustration we perform this calculation when $q=2$.  By expanding (3) the following formula holds (see [D4] for details and precise formulas
for the $n_j$):

$$  f_{p,2}(x,y) = \Phi_{\Gamma(p,2)} (z,{\overline z}) = x^p + (-1)^{p+1} y^p + \sum_j n_j x^{p-2j} y^{j}.  \eqno (14.1) $$
Here the $n_j$ are positive integers and the summation index $j$ satisfies $2j \le p$.
The target hyperquadric now depends on whether $p$ is even or odd. When $p=2r-1$ is odd,
the target hyperquadric is the sphere, namely the hyperquadric $Q(r+1,0)$.
When $p=2r$ is even, the target hyperquadric is $Q(r+1,1)$.
In any case, using (7) under the condition $x +\sqrt{x^2+4y} > 2y$, we obtain

$$  \left(f_{p,2}(x,y)\right)^{1 \over p} = {x + \sqrt{x^2 + 4y} \over 2} (1 + h_p(x,y))^{1 \over p}, \eqno (14.2) $$ 
where $h_p(x,y)$ tends to zero as $p$ tends to infinity. Note that we recover Corollary 2.1 by setting $x=y=1$. We summarize
this example in the following result. Similar results hold for the $f_{p,q}$ for $q \ge 3$. See [D4].

\begin{proposition} For $x +\sqrt{x^2+4y} > 2y$, the limit, as $p$ tends to infinity, of the left-hand side of (14.2) exists and equals 
${x + \sqrt{x^2 + 4y} \over 2}$.
\end{proposition}

It is also possible to compute $\Phi_{\Gamma_{(p,p-1)}}$ exactly. After some computation we obtain the following:

$$ \Phi_\Gamma (z,{\overline z}) = |z_1|^{2p} + |z_2|^{2p} + \sum_j n_j (|z_1|^2 |z_2|^2)^j = x^p + y^p + \sum n_j (xy)^j, \eqno (15) $$
where the $n_j$ are integers. They are $0$ when $2j > p$, and otherwise non-zero.
In this range $n_j > 0$ when $j$ is odd, and $n_j < 0$ when $j$ is even.
Explicit formulas for the $n_j$ exist; in fact they are closely related to the coefficients for $f_{p,2}$. See [D3].
To see what is going on, we must consider the four possibilities for $p$ modulo $(4)$.

We illustrate by listing the polynomials of degrees $4,5,6,7$. As above
we put $|z_1|^2 = x$ and $|z_2|^2 = y$. We obtain:

$$  f_{4,3}(x,y) = x^4 + y^4+ 4xy - 2x^2 y^2 \eqno (16.4) $$

$$  f_{5,4}(x,y) = x^5 + y^5+ 5xy - 5x^2y^2 \eqno (16.5) $$

$$  f_{6,5}(x,y) = x^6 + y^6+ 6xy - 9x^2y^2 + 2 x^3 y^3 \eqno (16.6) $$

$$  f_{7,6}(x,y) = x^7 + y^7+ 7xy - 14 x^2y^2 + 7 x^3 y^3. \eqno (16.7) $$

For $\Gamma(p,p-1)$
one can show the following. When $p=4k$ or $p=4k+1$, we have $k+2$ positive coefficients and $k$ negative coefficients.
When $p=4k+2$ or $p=4k+3$, we have $k+3$ positive coefficients and $k$ negative coefficients.
For $q>2$ in general one obtains some negative coefficients when expanding $f_{p,q}$, and hence the target must be a (non-spherical) hyperquadric.
The paper [LWW] provides a method for determining the sign of the coefficients.

Given a finite subgroup $\Gamma$ of $U(n)$, 
the invariant polynomial $\Phi_\Gamma$ is Hermitian symmetric,
and hence its underlying matrix of coefficients is Hermitian. We let $N_+(\Gamma)$ denote the number of
positive eigenvalues of this matrix, and we let $N_-(\Gamma)$ denote the number of negative eigenvalues.
When $\Gamma$ is cyclic of order $p$ we sometimes write $N_+(p)$ instead of $N_+(\Gamma)$, but the reader should
be warned that the numbers $N_+$ and $N_-$ depend upon $\Gamma$ and not just $p$. 
The ratio $R_p = { N_+(p) \over N_+(p) + N_-(p)}$ is of some interest, but it can be hard to compute.
We therefore consider its asymptotic behavior.

For the class of groups considered above Theorem 4.1 holds. It is a special case of a result to appear 
in the doctoral thesis  [G] of Grundmeier, who has found the limit of $R_p$ for many classes of groups (not necessarily cyclic)
whose order depends on $p$. Many different limiting values can occur.
Here we state only the following simple version which applies to the three classes under consideration.

\begin{proposition} For the three classes of cyclic groups whose invariant polynomials
are given by (13), (14), and (15), the limit of $R_p$
as $p$ tends to infinity exists. In the first two cases the limit is $1$. When $\Phi_\Gamma$ satisfies (15),
the limit is ${1 \over 2}$. \end{proposition}

\begin{remark} For the class of groups $\Gamma(p,q)$ the limit $L_q$ of $R_p$ exists and depends on $q$.
If one then lets $q$ tend to infinity, the resulting limit equals ${3 \over 4}$. Thus the asymptotic result
differs from the limit obtained by setting $q=p-1$ at the start. The subtlety of the situation is evident.
\end{remark}

\section{Metacyclic groups}

Let $C_p$ denote a cyclic group of order $p$.
A group $G$ is called {\it metacyclic} if there is an exact sequence of the form

$$ 1 \to C_p \to G \to C_q \to 1.$$
Such groups are also described in terms of two generators $A$ and $B$ such that $A^p=I$, $B^q = I$, and
$AB=B^m A$ for some $m$. In this section we will consider metacyclic subgroups of $U(2)$ defined as follows.
Let $\omega$ be a primitive $p$-th root of unity, and let $A$ be the following element of $U(2)$:

$$ 
\begin{pmatrix}
\omega     & 0 \\
0 & {\overline \omega}
\end{pmatrix}.
 \eqno (17) $$

For these metacyclic groups we obtain in (23) a formula for the invariant polynomials in terms of known
invariant polynomials for cyclic groups. We write $C(p,p-1)$ for the cyclic subgroup of $U(2)$ generated by $A$. Its invariant polynomial is
$$ \Phi_{C(p,p-1)} = 1- \prod_{k=0}^p (1 - \langle A^k z, z \rangle). \eqno (18)$$ 

Now return to the metacyclic group $\Gamma$.
Each group element of $\Gamma$ will be of the form $B^j A^k$ for appropriate exponents $j,k$. Since $B$ is unitary,
$B^* = B^{-1}$. We may therefore write

$$ \langle B^j A^k z, w \rangle = \langle A^k z,  B^{-j} w \rangle. \eqno (19) $$
We use (19) in the product
defining $\Phi_\Gamma$ to obtain the following formula:

$$ \Phi_\Gamma(z,{\overline z}) = 1 - \prod_{k=0}^{p-1} \prod_{j=0}^{q-1} (1 - \langle B^j A^k z, z \rangle)
= 1- \prod_{k=0}^{p-1} \prod_{j=0}^{q-1} (1 - \langle A^k z, B^{-j} z \rangle). \eqno (20) $$
Notice that the term 

$$ \prod_{k=0}^{p-1} (1 - \langle A^k z, B^{-j} z \rangle) \eqno (21) $$
can be expressed in terms of the invariant polynomial for the cyclic group $C(p,p-1)$.
We have

$$ \prod_{k=0}^{p-1} (1 - \langle A^k z, B^{-j} z \rangle) = 1 - \Phi_{C(p,p-1)}(z, B^{-j}z), \eqno (22) $$
and hence we obtain

$$ \Phi_\Gamma(z,{\overline z}) = 1- \prod_{j=0}^{q-1} \left( 1 - \Phi_{C(p,p-1)}(z, B^{-j}z) \right).  \eqno (23) $$
The invariance of $\Phi_\Gamma$ follows from the definition, but this property
is not immediately evident from this polarized formula. The other properties from Theorem 1.1 are evident in this version of the formula. 
We have $\Phi_\Gamma(0,0)=0$. Also,
$\Phi_\Gamma(z, {\overline z}) = 1$ on the unit sphere, because of the term when $j=0$. The degree in $z$ is $pq$
because we have a product of $q$ terms each of degree $p$. 

The simplest examples of metacyclic groups are the dihedral groups.
The dihedral group $D_p$ is the group of symmetries of a regular polygon of $p$ sides.
The group $D_p$ has order $2p$; it is generated by two elements $A$ and $B$, which
satisfy the relations $A^p=I$, $B^2=I$, and $AB=BA^{p-1}$. Thus $A$ corresponds
to a rotation and $B$ corresponds to a reflection. We may represent $D_p$ as a subgroup
of $U(2)$ by putting 

$$ 
A= \begin{pmatrix}
\omega     & 0 \\
0 & {\omega^{-1}}
\end{pmatrix}
 \eqno (24.1 ) $$

$$ 
B= \begin{pmatrix}
0 & 1 \\
1 & 0
\end{pmatrix}.
 \eqno (24.2) $$

Formula (23) for the invariant polynomial simplifies because the product in (23) has only two terms. We obtain the following result,
proved earlier in [D2].

\begin{theorem} The invariant polynomial for the above representation of $D_p$
satisfies the following formula:

$$ \Phi(z,{\overline z}) = f_{p,p-1}(|z_1|^2, |z_2|^2) + f_{p,p-1}(z_2 {\overline
z_1}, z_1 {\overline z_2} ) -
f_{p,p-1}(|z_1|^2, |z_2|^2) f_{p,p-1}(z_2 {\overline
z_1}, z_1 {\overline z_2} ).  \eqno (25) $$
\end{theorem}

\section{An application; failure of rigidity}

In this section we use the group invariant approach
to construct the first examples of polynomial mappings of degree $2p$ from $Q(2,2p+1)$ to $Q(N(p), 2p+1)$.
The key point of these examples is that the number of negative eigenvalues is preserved. The mappings
illustrate the failure of rigidity in the case where we keep the number of negative eigenvalues
the same but we are allowed to increase the number of positive eigenvalues
sufficiently. The mappings arise from part of a general theory being developed [DLe2] by the author and J. Lebl.
As mentioned in the introduction, the additional assumption that the mapping preserves sides of the hyperquadric
does force linearity in this context. [BH]

\begin{theorem} Let $2p+1$ be an odd number with $p\ge 1$. There is an integer $N(p)$ and a holomorphic polynomial mapping
$g_p$ of degree $2p$ such that

$$ g_p : Q(2,2p+1) \to Q(N(p), 2p+1). $$
and $g_p$ maps to no hyperquadric with smaller numbers of positive or negative eigenvalues.
\end{theorem}

\begin{proof} We begin with the group $\Gamma(2p,2)$. We expand the formula given in (7) with $p$ replaced by $2p$. The result
is a polynomial $f_{2p,2}$ in the two variables $x,y$ with the following properties. First, the coefficients are positive except
for the coefficient of $y^{2p}$ which is $-1$. Second, we have $f_{2p,2}(x,y)=1$ on $x+y=1$. Third, because of the group invariance,
only even powers of $x$ arise. We therefore can replace $x$ by $-x$ and obtain a polynomial $f(x,y)$ such that
$f(x,y)=1$ on $-x+y=1$ and again, all coefficients are positive except for the coefficient of $y^{2p}$.
Next replace $y$ by $Y_1 + Y_2$. We obtain a polynomial in $x,Y_1,Y_2$ which has precisely $2p+1$ terms with negative coefficients.
These terms arise from expanding $-(Y_1+Y_2)^{2p}$. All other terms have positive coefficients. This polynomial takes the value
$1$ on the set $-x + Y_1+Y_2=1$. Now replace $x$ by $X_1 + ... + X_{2p+1}$.

We now have a polynomial $W(X,Y)$ that is $1$ on the set given by
$$ -\sum_1^{2p+1} X_j + \sum_1^2 Y_j= 1. $$
It has precisely $2p+1$ terms with negative coefficients.
There are many terms with positive coefficients; suppose that the number is $N(p)$.
In order to get back to the holomorphic setting, 
we put $X_j = |z_j|^2$ for $1 \le j \le 2p+1$ and we put $Y_1 = |z_{2p+2}|^2$ and $Y_2 = |z_{2p+3}|^2$.
We note that this idea (an example of the moment map)
has been often used in this paper, as well as in the author's work on proper mappings between balls; see for example [D1] and [DKR].
Let $g_p(z)$ be the mapping, determined up to a diagonal unitary matrix, with

$$ \sum_{j=1}^{N(p)} |g_j(z)|^2 - \sum_{j=1}^{2p+1} |g_j(z)|^2 = W(X,Y). \eqno (26) $$
Each component of $g_p$ is determined by (26) up to a complex number of modulus $1$.
The degree of $g_p$ is the same as the degree of $w$. We obtain all the claimed properties. \end{proof}

\begin{example} We write out everything explicitly when $p=1$. Let $c=\sqrt{2}$.
The proof of Theorem 6.1 yields the polynomial mapping $g:Q(2,3) \to Q(8,3)$ of degree $2$ defined by

$$ g(z) =  (z_1^2,z_2^2,z_3^2, c z_1z_2,c z_1 z_3, c z_2 z_3, c z_4, c z_5; z_4^2, c z_4 z_5, z_5^2).   \eqno (27) $$
Notice that we used a semi-colon after the first eight terms to highlight that $g$ maps to $Q(8,3)$.
Summing the squared moduli of the first eight terms yields
$$ (|z_1|^2+|z_2|^2 + |z_3|^2)^2 + 2( |z_4|^2 + |z_5|^2). \eqno (28) $$
Summing the squared moduli of the last three terms yields
$$ (|z_4|^2 + |z_5|^2)^2. \eqno (29) $$
The set $Q(2,3)$ is given by 
$$  |z_4|^2 + |z_5|^2 -1 = |z_1|^2 + |z_2|^2 +|z_3|^2. $$
On this set we obtain $1$ when we subtract (29) from (28).
\end{example}

\section*{Bibliography}

[BER] Baouendi, M. Salah, Ebenfelt, Peter, and Rothschild, Linda Preiss,
Real submanifolds in complex space and their mappings, Princeton
Mathematical Series 47, Princeton University Press, Princeton, NJ, 1999.

[BH] Baouendi, M. S.and Huang, X., Super-rigidity for holomorphic mappings between hyperquadrics with positive signature,
J. Differential Geom.  69  (2005),  no. 2, 379-398. 

[D1] D'Angelo, J. P., Several Complex Variables and the Geometry of Real Hypersurfaces, CRC Press, Boca Raton, 1992. 

[D2] D'Angelo, J. P., Number-theoretic properties of certain CR mappings, J. of Geometric Analysis,
Vol. 14, Number 2 (2004), 215-229.

[D3] D'Angelo, J. P.,  Invariant holomorphic maps, J. of Geometric Analysis, Vol. 6, 1996, 163-179.

[D4] D'Angelo, J. P., Asymptotics for invariant polynomials, Complex Variables and Elliptic Equations,
Vol. 52, No. 4 (2007), 261-272.

[DKR] D'Angelo, J. P., Kos, \v Simon, and Riehl E., A sharp bound for the degree of proper monomial
mappings between balls, J. of Geometric Analysis, Vol. 13, Number 4 (2003), 581-593. 

[DL] D'Angelo, J. P.  and Lichtblau, D. A., Spherical space forms, CR maps,
and proper maps between balls, J. of Geometric Analysis, Vol. 2, No. 5(1992), 391-416.

[DLe] D'Angelo, J. P.  and Lebl,J.,  Complexity results for CR mappings between spheres, International J. Math, Vol. 20, No. 2 (2009), 1-18.

[DLe2] D'Angelo, J. P.  and Lebl, J. (work in progress).

[F1] Forstneric, F., Extending proper holomorphic mappings of positive codimension,  Invent. Math.  95  (1989),  no. 1, 31-61. 

[F2] Forstneric, F., Proper holomorphic mappings: a survey. Several complex variables (Stockholm, 1987/1988), 297-363, Math. Notes, 38, Princeton Univ. Press, Princeton, NJ, 1993. 

[G] Grundmeier, D., Doctoral thesis, University of Illinois, expected 2010.

[Li1] Lichtblau D. A., Invariant proper holomorphic maps between balls. Indiana Univ. Math. J. 41 (1992), no. 1, 213-231. 

[Li2] Lichtblau, D. A., Invariant proper holomorphic maps between balls, PhD thesis, University of Illinois, 1991.

[LWW] Loehr, N. A. , Warrington, G. S.,  and Wilf, H. S., The combinatorics of a three-line circulant determinant,
Israel J. Math. 143 (2004), 141-156.

[S] Smith, Larry, Polynomial Invariants of Finite Groups, Research Notes in Mathematics, 6. A K Peters, Ltd., Wellesley, MA, 1995.

\end{document}